\documentclass{amsproc}
\usepackage{mathrsfs}
\usepackage[all,cmtip]{xy}
\usepackage{tikz}
\usepackage{cancel}

\newtheorem{theorem}{Theorem}[section]
\newtheorem{lemma}[theorem]{Lemma}

\theoremstyle{definition}

\newtheorem{proposition}[theorem]{Proposition}

\theoremstyle{remark}
\newtheorem{remark}[theorem]{Remark}

\numberwithin{equation}{section}

\begin{document}

\title{Complex surfaces with mutually \\non-biholomorphic universal covers}

\author{Gabino Gonz\'alez-Diez}
\address{Departamento de Matem\'aticas, Universidad Aut\'onoma de Madrid, Spain.}
\email{gabino.gonzalez@uam.es}
\thanks{Partially supported by Spanish Government Research Project MTM2016-79497.}

\author{Sebasti\'an Reyes-Carocca}
\address{Departamento de Matem\'atica y Estad\'istica, Universidad de La Frontera, Chile.}
\email{sebastian.reyes@ufrontera.cl}
\thanks{Partially supported by Fondecyt Grant 11180024, 1190991 and Redes Grant 2017-170071}

\date{xx.}



\begin{abstract}
It is known that the universal cover of compact Riemann surface is either the projective line, the complex plane or the unit disk. In this article we construct a very explicit  family of complex surfaces that gives rise to uncountably many mutually non-biholomorphic universal covers. The slope of these surfaces, which are going to be total spaces of Kodaira fibrations, is also determined.\end{abstract}

\maketitle

%

\section{Introduction and statement of results}
Let $X$ be a complex  projective  variety and let $k$ be a subfield of the field of the complex numbers. We shall say that $X$ is defined over $k$ or that $k$ is a {\it field of definition} for $X$ if there exists a collection of homogenous polynomials with coefficients in $k$ so that the variety  they define is isomorphic to $X.$ 

\vspace{0.2cm}

Let $S$ be a non-singular minimal projective surface of general type. Based on results due to Bers  and Griffiths \cite{Bers, Griffiths} on  uniformization of complex projective varieties and on results of Shabat \cite{Sha, Shabat} on automorphism groups of universal covers of families of Riemann surfaces, in \cite{th1} the authors succeeded in proving that whether or not $S$ is defined over a given algebraically closed field depends only on the holomorphic universal cover of its Zariski open subsets.

\vspace{0.2cm}

A special kind of complex surfaces as before are the total space of the so-called {\it Kodaira fibrations}. A Kodaira fibration consists of a  compact complex surface $S,$ a compact Riemann surface $B$ and a surjective holomorphic map $S \to B$ everywhere of maximal rank such that the fibers are connected and not mutually isomorphic Riemann surfaces. The genus $g$ of the fibre is called the genus of the fibration and it is known that  necessarily $g\ge 3.$ 

\vspace{0.2cm}

Kodaira fibrations were introduced  by Kodaira himself in \cite{Kodaira} to show that the signature of a differentiable fiber bundle need not be multiplicative. Further properties of Kodaira fibrations were obtained by  Atiyah, Catanese, Hirzebruch and Kas,  among others; see, for example, \cite{Atiyah, Catanese2, Hirzebruch, Kas}.
 For explicit constructions of Kodaira fibrations we refer to the articles  \cite{BD, Catanese, gabkodaira, HG, Zaal}.

\vspace{0.2cm}

Let $S \to B$ be a Kodaira fibration. According to  \cite[Theorem 8]{th1} (together with \cite[Remark 9]{th1})  the relation between fields of definition and universal covers takes a neater form in this case. For these surfaces whether or not
$S$ is defined over a given algebraically closed field depends only on the holomorphic universal cover of S instead of on the whole collection of Zariski open subsets. In a more concrete way, for $i=1,2$, let $S_i \to B_i$  be a Kodaira fibration. In \cite[Theorem 1.1]{CMH} it was proved that if the holomorphic universal covers of  $S_1$ and $S_2$ are biholomorphically equivalent, then $S_1$ and $S_2$ are defined over the same algebraically closed fields. 

\vspace{0.2cm}

In this paper, for any integer $g \ge 7,$ we  construct a collection of Kodaira fibrations of genus $g$ and then we apply \cite[Theorem 1.1]{CMH} to show that it contains an uncountable subfamily whose corresponding holomorphic universal covers are mutually non-biholomorphic. 

\begin{theorem} 
For any $\lambda  \in \mathbb{C} - \{0, -\tfrac{27}{4}\} $ and $r\ge 8$ even, there is a Kodaira fibration $S_{\lambda, r} \to B$ of genus $g=3+\tfrac{r}{2}$ such that if $\lambda$ and $\mu$ are algebraically independent transcendental complex numbers, then the universal covers of $S_{\lambda, r}$ and $S_{\mu, r}$ are non-biholomorphically equivalent bounded domains of $\mathbb{C}^2.$

In particular, for each $r \ge 8$ even, there are in this family uncountably many Kodaira fibrations whose corresponding holomorphic universal covers are pairwise non-biholomorphically equivalent. 
\end{theorem}

As pointed out in \cite{CMH} and \cite{th1}, we remark that, in this respect,  complex surfaces are very much in contrast with compact Riemann surfaces, for which the universal cover only depends on the genus.

\vspace{0.2cm}

As mentioned above, Kodaira fibrations were the first known examples of fiber bundles whose signature is not multiplicative. We recall that the signature $\tau(S)$ of an algebraic surface $S$ is related to the ratio of its Chern numbers $$\upsilon(S):=c_1^2(S)/c_2(S)$$by the formula $$\tau(S)=e(S)(\upsilon(S)-2)/3,$$ where $e(S)$ stands for the Euler characteristic of $S$. Thus, the vanishing of the signature is equivalent to the {\it slope} $\upsilon(S)$ being equal to 2. In the  case of our surfaces $S_{\lambda, r}$ we have:

\begin{proposition}
$$\upsilon(S_{\lambda, r})=2+\frac{3}{2(4+r)}.$$
\end{proposition}

It is known that the slope of a Kodaira surface lies in the interval $(2, 3)$  and that the limit values $2$ and $3$ correspond to surfaces whose universal cover is the bidisk and the $2$-ball respectively (see, for example, \cite[Section 2.1.2]{Catanese2}).  From this point of view, the fact that the limit of $\upsilon(S_{\lambda,r})$ tends to 2 when $r$ tends to infinity could be interpreted as an indication that the universal covers of the surfaces $S_{\lambda,r}$, which are known to be bounded domains of $\mathbb{C}^2$, approach the bidisk when the   integer $r$ is allowed to take large values. As far as we know, the examples with largest slope have been constructed by Catanese and Rollenske \cite{Catanese} and reach the value $2+2/3.$

{\bf Acknowledgments.} The construction of these families goes back to old work by William J. Harvey and the first author.
William J. Harvey is the supervisor of the first author and the supervisor’s supervisor of the second. We would like to take this opportunity to express our hearty  gratitude to him.


\section{Preliminaries}
\subsection{Uniformization of Kodaira surfaces} 

Let $f: S \to B$ be a Kodaira fibration, let $\pi : \mathbb{H} \to B$ be the universal covering map of $B$ and let $\Gamma$ be the covering group so that $B \cong \mathbb{H}/ \Gamma.$ By considering the pull-back $$h: \pi^*S \to \mathbb{H}$$ of $f$ by $\pi,$ we obtain a new family of Riemann surfaces in which, for each $t \in \mathbb{H},$ the fiber $h^{-1}(t)$ agrees with the Riemann surface $f^{-1}(\pi(t)).$

Teichm\"{u}ller theory together with results due to Bers and Griffiths \cite{Bers, Griffiths} enable us to choose uniformizations $h^{-1}(t)=D_t/K_t$ possessing the following properties:
\begin{itemize} 
\item[(a)] $K_t$ is a Kleinian group acting on a bounded domain $D_t$ of $\mathbb{C}$ which is biholomorphically equivalent to a disk. 
\item[(b)] The union of all these disks $$\mathscr{B}:=\cup_{t \in \mathbb{H}} D_t$$is a bounded domain of $\mathbb{C}^2$ 
which is biholomorphic to the universal cover of $S,$ that is, $S \cong \mathscr{B}/\mathbb{G},$ where $\mathbb{G} < \mbox{Aut}(\mathscr{B})$ is the  covering group.
\item[(c)] The group $\mathbb{G}$ is endowed with a surjective homomorphism of groups $\Theta : \mathbb{G} \to \Gamma$ which induces an exact sequence of groups \[ \xymatrix {
  1 \ar[r] &
  \mathbb{K} \ar[r] &
  \mathbb{G} \ar[r]^{\Theta} &
  \Gamma \ar[r] &
  1,
} \]where the subgroup $\mathbb{K}$ preserves each {\it quasi-disk} $D_t$ and acts on it as the Kleinian group $K_t$ for each $t \in \mathbb{H}.$
\end{itemize}

We note that $\mathscr{B}$ carries a fibration structure itself $\mathscr{B} \to \mathbb{H}$ whose fiber over $t \in \mathbb{H}$ is $D_t$ (the domain $\mathscr{B}$   is called a {\it Bergman domain} in Bers' terminology; see \cite[p. 284]{Bers}). The situation is summarized in the following commutative diagram
$$
\begin{tikzpicture}[node distance=2.5 cm, auto]
  \node (P) {$\pi^*S$};
  \node (Q) [right of=P] {$S \cong \mathscr{B}/\mathbb{G}$};
  \node (A) [below of=P, node distance=1.2 cm] {$\mathbb{H}$};
  \node (C) [below of=Q, node distance=1.2 cm] {$B \cong \mathbb{H} / \Gamma$};
  \draw[->] (P) to node {} (Q);
  \draw[->] (A) to node {$\pi$} (C);
  \draw[->] (P) to node [swap] {$h$} (A);
    \draw[->] (Q) to node {$f$} (C);
\end{tikzpicture}
$$

\subsection{Basic Teichm\"{u}ller theory}
Let $C_0$ be a compact Riemann surface of genus $g$ admitting an automorphism 
$\tau_0$ 
of order two with $r$ fixed points such that $C_0/ \langle \tau_0 \rangle$ has genus two. By the Riemann-Hurwitz formula, $r$ is necessarily even and  $$g=3+\tfrac{r}{2}.$$ 

We denote by $R_0^*$ the Riemann surface obtained from $R_0:=C_0 / \langle \tau_0 \rangle$ by removing the $r$ branch values $\{q_1, \ldots, q_r\}$ of the associated covering map $C_0 \to R_0.$ Let us denote by $\mbox{Mod}_g(\tau_0)$ the normalizer of $\langle \tau_0 \rangle$ viewed as a subgroup of the mapping class group $$\mbox{Mod}(C_0)=\mbox{Mod}_{g}.$$ 

The Teichm\"{u}ller space $T_{2,r}$ of the Riemann surface $R_0^*$ is naturally embedded in the Teichm\"{u}ller space $T_g$ of $C_0$ as the fixed locus of $\langle \tau_0 \rangle.$ Therefore, the group $\mbox{Mod}_g(\tau_0)$ acts on $T_{2,r}$ in a obvious way, and this action induces a faithful action of $\mbox{Mod}_g(\tau_0)/\langle \tau_0 \rangle$ on $T_{2,r}.$ This permits to realize $\mbox{Mod}_g(\tau_0)/\langle \tau_0 \rangle$ as a finite index subgroup of the full group of holomorphic automorphisms  of $T_{2,r}$  which is the mapping class group $$\mbox{Mod}(R_0^*)=\mbox{Mod}_{2,r}.$$

In a similar way, we denote by $\mathscr{P}_{2,r}$ the {\it pure mapping class group}, which is the subgroup of $\mbox{Mod}_{2,r}$ consisting of those mapping classes that preserve each distinguished point $q_i$ for $1 \le i \le r.$ Analogously, we denote by $\mathscr{P}_g(\tau_0)$ the {\it pure relative modular group}, which is the subgroup of $\mbox{Mod}_g(\tau_0)$ consisting of those mapping classes that preserve each fixed point of $\tau_0.$ 

We recall  that just as the quotient  $$\mathscr{M}_{g,n} =T_{g,n}/\mbox{Mod}_{g,n} $$yields the moduli space of Riemann surfaces of genus $g$ with a set of $n$ distinguished points, the quotient $$\mathscr{M}^{pure}_{g,n} =T_{g,n}/\mathscr{P}_{g,n}$$yields the moduli space of Riemann surfaces of genus $g$ with an ordered set of $n$ distinguished points. In other words, two $(n+1)$-tuples $(C;p_1, \ldots, p_n)$ and $(D;q_1, \ldots, q_n)$ represent the same point in $\mathscr{M}_{g,n}$ if and only if there is an isomorphism $\alpha:C\to D$ which sends the set $\{p_1, \ldots, p_n\}$ to the set $\{q_1, \ldots, q_n\}$ whereas they represent the same point in 
$  \mathscr{M}^{pure}_{g,n}$ only if $\alpha$ sends $p_i$ exactly to $q_i$ for $ i=1,\ldots, n.$

Next, we introduce the {\it level three} normal subgroup $\mbox{Mod}_g[3]$ of $\mbox{Mod}_g$ consisting of those (homotopy classes of) homeomorphisms of $C_0$ which induce the identity map on the  homology group $\mbox{H}_1(C_0, \mathbb{Z}/3\mathbb{Z})$. Since by Serre's lemma the group $\mbox{Mod}_g[3]$ does not contain any non-trivial element of finite order and since it has finite index in $\mbox{Mod}_g,$ the homomorphism $\mathscr{P}_g(\tau_0) \to \mathscr{P}_{2,r}$ maps injectively the group $$\mathscr{P}_g(\tau_0) \cap \mbox{Mod}_g[3]$$ into a finite index subgroup of $\mathscr{P}_{2,r}$, henceforth denoted by $\mathscr{P}_g(3,\tau_0).$  See \cite[Proposition 1.1]{HGN}.

We refer to the article \cite{HG2} for more details.

\section{Proof of the Theorem}

\subsection{Construction of the surfaces $S_{\lambda, r}$ }

Let $\lambda$ be 
an arbitrary complex number different from $0$ and $-\tfrac{27}{4}.$ Consider the (affine) elliptic curve $$E_{\lambda} \, : \,  y^2=x^3+\lambda x + \lambda$$ endowed with the standard group law $\oplus$, with the point $\infty$ at infinity playing the role of neutral element. The map $\pi : (x,y) \mapsto (x^2,y)$ induces a two-fold covering map between the genus two Riemann surface $$X_{\lambda} \, : \,  y^2=x^6+\lambda x^2+ \lambda$$and $E_{\lambda}$ with branch points 
$s_{+}=(0, \lambda^{1/2})$ and $s_{-}=(0, - \lambda^{1/2}).$ 

Clearly, we can inductively construct a sequence  of      $\overline{\mathbb{Q}(\lambda)}-$rational points 
$$e_{2}, \ldots, e_r, \ldots \in E_{\lambda}$$ 
such that 
all the points $e_i$  in this sequence as well as their differences $e_i \ominus e_j$  
are different from the points
$\pi(s_{+}) \ominus \pi(s_{-})$,   its inverse
$\pi(s_{-}) \ominus \pi(s_{+})$ and  the identity element $\infty \in E_{\lambda}$. 
 Now, for any natural number $r \geq 1$ we consider the complex one-dimensional analytic space
$$C_{\lambda,r}:=\{(p_1, \ldots, p_r) \in X_{\lambda}^r : \pi(p_i)=\pi(p_1) \oplus e_{i} \mbox{ for each } 2 \le i \le r\}.
$$Note that $$C_{\lambda, r} \subset X_{\lambda}^{r} - \Delta_r$$where $\Delta_r$ stands for the diagonal subset of $X_{\lambda}^r$ consisting of the $r$-tuples with two coincident coordinates.

\vspace{0.2cm}

{\bf Claim.} $C_{\lambda, r}$ is a compact Riemann surface of genus $r2^{r-1}+1.$

\vspace{0.2cm}

Let $\Psi : X_{\lambda}^r \to E^{r-1}$ be the holomorphic map defined by 
$$(p_1, \ldots, p_r) \mapsto (\pi(p_2) \ominus \pi(p_1) \ominus e_2, \ldots, \pi(p_r) \ominus \pi(p_1) \ominus e_r).$$  
It is clear that $C_{\lambda, r}=\Psi^{-1}(\infty, \ldots, \infty).$ 
An easy computation shows that the Jacobian matrix of $\Psi$ at a point 
$(p_1, \ldots, p_r)$ is of the form   
$$
 J=\frac{\partial (\Psi_1,  \ldots, \Psi_{r-1})}{\partial (z_1, \ldots,  z_r)}=\left(
\begin{array}
[c]{cccccc}%
-\pi'(p_1) & \pi'(p_2) & 0 & \cdots & 0 &0\\
-\pi'(p_1) & 0 & \pi'(p_3) & \cdots & 0 & 0\\
 \,\, \,\,\vdots & \vdots & \vdots  & \ddots & \vdots & \vdots\\
-\pi'(p_1) & 0 & 0 & \cdots & \pi'(p_{r-1}) & 0\\
-\pi'(p_1) & 0 & 0 & \cdots & 0 & \pi'(p_r)
\end{array}
\right)
$$

 Now, $\pi'(p_{i})=0$ if and only if $p_{i}$ is one the branch points $s_{+}, s_{-}$. But the choice of the points $e_i$
ensures that at most one of the entries of  $(p_1, \ldots, p_r)$ can be a branch point.   Therefore 
  $J$
 has maximal rank $r-1$ at each point of $C_{\lambda,r},$   showing that   $C_{\lambda, r}$ is a (possibly non-connected) non-singular complex curve. Now,
the map 
$$\pi_{\lambda, r} : (p_1, \ldots, p_r) \mapsto (p_1, \ldots, p_{r-1})$$ is  clearly a holomorphic two-fold ramified  cover between $C_{\lambda, r}$ and $C_{\lambda, r-1}$. This implies that  $C_{\lambda, r}$ is a connected Riemann surface. Moreover, it is easy to see that this double cover has exactly  $2^{r}$ branch points, namely all points $ (p_1, \ldots, p_r)$ with $p_r=s_{+}$ or $s_{-}.$ From here the value of the genus follows 
from the Riemann-Hurwitz formula and induction on $r$. This completes the proof of our claim.

\vspace{0.2cm}

Now the strategy is to associate 
 to each point $(p_1, \ldots, p_r) \in C_{\lambda, r}$ a two-fold branched cover of $X_{\lambda}$ ramified over $p_1, \ldots, p_r.$ But, in order to show that this idea leads to a well-defined Kodaira fibration, we must proceed carefully. 
 
\vspace{0.2cm}

Let $\mathbb{H} \to C_{\lambda, r}$  
be the universal cover of $C_{\lambda, r}$ and $\Gamma$ its covering group so that $C_{\lambda, r} \cong \mathbb{H}/ \Gamma.$ The map $$X_{\lambda}^r-\Delta_r \to \mathscr{M}_{2,r}^{pure} \,\, \mbox{ defined by }\,\, (p_1, \ldots, p_r) \mapsto [X_{\lambda}, \{p_1, \ldots, p_r\}]$$ is a holomorphic map \cite[Lemma 1]{HG}; let us denote by $\Phi$ its restriction to $C_{\lambda, r}.$ Following \cite[Proposition 3.1]{HGN}, if $r \ge 7$ then $\mathscr{P}_{2,r}$ acts freely on $T_{2,r}$ and therefore, by the theory of covering spaces, there exists a holomorphic lift $$\hat{\Phi}: \mathbb{H} \to T_{2,r} \,\,\, \mbox{ of } \,\,\, \Phi : C_{\lambda, r} \to \mathscr{M}_{2,r}^{pure}$$ together with a group  homomorphism  $\Theta : \Gamma \to \mathscr{P}_{2,r}$ 
defined by 
$$\Theta({ \alpha}) \circ \hat{\Phi}=\hat{\Phi} \circ { \alpha} \, \, \mbox{ with } \,\, { \alpha } \in \Gamma.$$

Let us now
set $$\Gamma[3]:=\Theta^{-1}(\mathscr{P}_g(3,\tau_0))$$and $C_{\lambda, r}[3]:=\mathbb{H}/\Gamma[3].$  Then, the natural projection $$\zeta_{\lambda, r} : C_{\lambda, r}[3] \to C_{\lambda, r}$$is a finite unbranched covering map of compact Riemann surfaces, and there is an obvious commutative diagram as follows:$$\begin{tikzpicture}[node distance=1.5 cm, auto]
  \node (A) [] {$C_{\lambda, r}[3]=\mathbb{H}/\Gamma[3]$};
    \node (C) [right of=A, node distance=4.5 cm] {$T_{2,r}/\mathscr{P}_g(3, \tau_0)$};
     \node (E) [below of=A] {$C_{\lambda, r}=\mathbb{H}/\Gamma$};
    \node (F) [below of=C] {$T_{2,r}/\mathscr{P}_{2,r}=\mathscr{M}^{pure}_{2,r}$};
    \node (G) [right of=C, node distance=3.5 cm] {$T_g/\mathscr{Mod}_g[3]$};
  
  \draw[->] (A) to node [swap] {$\zeta_{\lambda, r}$} (E);
        \draw[->] (A) to node {} (C);
  \draw[->] (C) to node {} (F);
  \draw[->] (E) to node {$\Phi$} (F);
    \draw[->] (C) to node {} (G);\end{tikzpicture}
$$
Let us denote by $\mathscr{M}_{g}[3]$ the  quotient space $T_g/\mbox{Mod}_g[3].$
This is called the \emph{moduli space with   level three} structure. Clearly, $\mathscr{M}_g[3]$ is a finite Galois cover of $\mathscr{M}_g$ with covering group $$\mbox{Mod}_g/\mbox{Mod}_g[3] \cong \mbox{Sp}(2g, \mathbb{Z}/3\mathbb{Z}).$$ 
Recall that the moduli space $\mathscr{M}_{g}$  comes equipped with a fibration $\pi_{g} :\mathscr{C}_{g} \to \mathscr{M}_{g}$ (called the \emph{universal curve}), whose fiber above a point $[F]\in \mathscr{M}_{g}$   is  a Riemann surface which is not  isomorphic to $F$ but
 to $F/\mbox{Aut}(F).$ 
However, the fact that $\mbox{Mod}_g[3]$ is torsion free,  implies that  
the corresponding level three universal curve $\pi_{g,3} : \mathscr{C}_g[3] \to \mathscr{M}_g[3]$ has the property that the fiber over a point
of $\mathscr{M}_g[3]$ parametrising a Riemann surface 
$F$ is isomorphic to the Riemann surface $F$ itself.  This means that for any non-constant holomorphic map $\varphi:B \to\mathscr{M}_{g}[3]$  the pull-back $ \varphi^*(\mathscr{C}_{g}[3])$ is a genus $g$ Kodaira fibration with base $B$. Thus, if   
 $$
 h_{\lambda, r}: C_{\lambda, r}[3] \to \mathscr{M}_g[3]
 $$
 denotes the holomorphic map defined by the first row of the above diagram then we get a Kodaira fibration a genus $g= 3+\frac{r}{2}$ 
$$f_{\lambda, r}:S_{\lambda, r}:=h_{\lambda, r}^* \mathscr{C}_{g}[3] \to B:=C_{\lambda, r}[3].$$
 simply by  considering the pull-back of $\pi_{g,3}$ by $h_{\lambda, r}.$
 
\vspace{0.2cm}

\begin{remark} \label{lapiz1} For later use, we observe that each 
 element $h \in \mbox{Mod}_g$ induces a fiber-preserving automorphism 
 of the universal curve $\mathscr{C}_g[3] \to \mathscr{M}_g[3]$. 
 When $h=\tau_0$ this automorphism, let us call it $\tau,$  preserves  each of the fibres 
 in the image of
 $T_{2,r}/\mathscr{P}_g(3, \tau_0)$. Therefore
  $\tau$ induces an automorphism of the 
 fibration $f_{\lambda, r}$ which restricts to an automorphism  $\tau_b$ on each fibre $$F_b:=f_{\lambda,r}^{-1}(b) \,\,\,  \mbox{ in such a way that }  \,\,\, F_b/ \langle \tau_b \rangle \cong X_{\lambda}.$$
\end{remark}

\subsection{Universal covers pairwise non-biholomorphic}
As the curve $C_{\lambda, r}$ is defined over $\overline{\mathbb{Q}(\lambda)},$ we can proceed analogously as in the proof of Theorem 4.1 in \cite{criterio1}  to ensure that curve $B$ is also defined over $\overline{\mathbb{Q}(\lambda)}.$   
Consequently, we
can 
apply Theorem 1.2 in \cite{CMH} to  infer that 
 $S_{\lambda, r} $ is also defined over $\overline{\mathbb{Q}(\lambda)}.$ 

  We now assume that $\lambda$ is a trascendental complex number. We claim that if $\mu$ is any  transcendental complex number which is algebraically independent of $\lambda$, then $S_{\lambda, r}$ cannot be defined over the field $\overline{\mathbb{Q}(\mu)}.$ Indeed,  if  that were the case, we could  apply Theorem 2.12 in \cite{criterio1} to deduce that $S_{\lambda, r}$
 is actually  defined over a number field and  this, in turn, would  imply that $E_{\lambda}$ is defined over a number field too (see \cite[Theorem 4.4]{criterio1}). But this is impossible since  the minimum field of definition of an elliptic curve is known to be the field generated by its $j-$invariant, which in this case is $$j(E_{\lambda})=\lambda^3/(4\lambda^3+27\lambda^2).$$  We now apply Theorem 1.1 in \cite{CMH} to prove the first statement of the theorem. The last part of the theorem follows directly from the existence of uncountably many  pairwise algebraically independent transcendental complex numbers.

\section{Proof of the Proposition } 

We recall that if $S$ is an algebraic surface, then $c_2(S)=e(S)$, the Euler class of $S$, and $$c_1^2(S)=K^2$$where $K^2$ stands for the self-intersection of the canonical divisor $K$ of $S.$ Moreover, if $S \to B$ is a Kodaira fibration of genus $g$ over a curve of genus $\gamma,$ it is known that $$e(S)=(2g-2)(2 \gamma -2).$$ Thus, for our 
surface $S_{\lambda, r}$ 
$$
\upsilon(S_{\lambda, r})={K^2}/{(2 \gamma-2)(4+r)}
$$and therefore to prove the equality
$$\upsilon(S_{\lambda, r})=2+\frac{3}{2(4+r)},$$we only need to check that 
 \begin{equation} \label{K22}
K^2 =(8+2r)(2\gamma-2)+3(\gamma-1).
\end{equation}


We begin by noting that Remark \ref{lapiz1} implies that, after passing to a smooth cover of the base if necessary, we can assume that $$S_{\lambda, r}/\langle \tau \rangle \cong B \times X_{\lambda}.$$

Let us consider the holomorphic map
$$g_{\lambda,r}:S_{\lambda, r} \to S_{\lambda, r}/ \langle \tau \rangle \to X_{\lambda}$$
induced by the projection $B \times X_{\lambda} \to X_{\lambda}$.
  Its ramification divisor $R$ is clearly   the fixed set of $\tau$. 
	By a further change 
of base (which does not alter the slope of the fibration) we can get a new Kodaira fibration, which we still denote  $f_{\lambda,r}:S_{\lambda, r} \to B,$ with the property that $R$ decomposes as a union of $r$ disjoint   sections    $R_1 \ldots, R_r$ of $f_{\lambda,r}$ such that  at each fibre  $F_b,$ the section 
$R_j$ picks out the $j$-th branch point of the restriction of 
$g_{\lambda,r}$ to $F_b$, that is to say the $j$-th fixed point of the automorphism $\tau_b$ (see Remark \ref{lapiz1}). 

Let us denote  by $\zeta_{\lambda, r}:B\to  C_{\lambda, r}$
the smooth cover obtained after the successive   changes of base performed on our initial Kodaira fibration and by 
 $\pi_j: C_{\lambda, r} \to X_{\lambda}$ the projection onto the $j$-th coordinate. Then  we see that 

\begin{enumerate}
\item [(i)]the restriction of $g_{\lambda,r}$ to $R_j$ is given by 
\begin{equation}\label{restrictiong}
g_{\lambda,r}|_{R_j} =\pi_j\circ \zeta_{\lambda, r} \circ f_{\lambda,r}|_{R_j}
\end{equation}
\item [(ii)] the restriction of $g_{\lambda,r}$ to a fibre $F_b$ is a degree $2$ covering 
$F_b \to X_{\lambda}$ whose branching value set is 
 $\{ \pi_1 \circ\zeta_{\lambda, r}(b), \cdots, \pi_r \circ \zeta_{\lambda, r}(b)  \}.$  
 \end{enumerate}
	 
Let us now consider the holomorphic forms on $X_{\lambda}$  
given by  $$w_1= xdx/y \,\,\,  \mbox{ and } \,\,\, w_2=dx/y$$
and let us denote by $D_1=s_1+t_1$ and $D_2=s_2+t_2$ their corresponding divisors. It is easily checked that $\{s_1, t_1\}=\{(0,\pm \lambda^{1/2})\}$ and  that $ s_2,t_2$ are the two points at infinity of $X_{\lambda}$.

Similarly we choose two holomorphic forms  $v_1$ and $v_2$ on $B$ with simple zeroes and denote their divisors as  
   $$\mbox{div}(v_k)= \Sigma_{i=1}^{2 \gamma -2}b_i^k.$$ 
   
   Without loss of generality, we can suppose that: 
   \begin{enumerate}
   \item[(a)] $\{b_i^1\}$ and  $\{b_i^2\}$ are disjoint sets.
   \item[(b)]  For no  of  triple $(k,i,j)$ with  $k=1,2$; $i=1, \ldots, 2\gamma-2$
	and $j=1,\ldots, =r$, the point 
$\pi_j \circ \zeta_{\lambda, r}(b_i^k)$ equals  
  $s_1$,  $t_1$, $s_2$ or $t_2.$  \newline In view of the observation (ii) above this is equivalent to saying that the branching values of  the restriction of $g_{\lambda,r}$  to the fibres  $F_{b_i^k}$ 
are different from $s_1$,  $t_1$, $s_2$ or $t_2.$
 \end{enumerate}

   Next we define the  following two sections of the canonical bundle  $K$
   $$
   W_k=f_{\lambda, r}^*v_k \wedge g_{\lambda,r}^*w_k;   \  	\   \   k=1,2.
   $$ 
    Clearly their corresponding divisors of {$S_{\lambda,r}$ can be expressed as  
    \begin{equation*} \label{ec1} 
\mbox{div}(W_k)=\Sigma_{i=1}^{2 \gamma -2} F_{b_i^k} + g_{ \lambda,r}^{-1}( D_k) + R
\end{equation*}

We can now compute 
$K^2$ as the intersection  $\mbox{div}(W_1) \cdot \mbox{div}(W_2)$ (see, for example,  \cite{Beau}). Setting 
$R=\Sigma_{j=1}^r R_j$ and 
taking into account that  $ \Sigma_{i=1}^{2 \gamma -2}b_i^1$  and  $ \Sigma_{i=1}^{2 \gamma -2}b_i^2$ (resp.
 $D_1$ and $D_2$) are disjoint divisors of $B$ (resp. of $X_{\lambda}$) we first get
$$ 
K^2= \left\{ \begin{array}{llllll}
 \cancel{\Sigma F_{b_i^1}\cdot \Sigma F_{b_i^2}} & + &
 \Sigma F_{b_i^1}\cdot g_{\lambda, r}^{-1}(D_2) &+ & \Sigma F_{b_i^1}\cdot \Sigma R_j & + \\
& & & & & \\
 g_{\lambda, r}^{-1}(D_1) \cdot \Sigma F_{b_i^2} &+ &
\cancel{ g_{\lambda, r}^{-1}(D_1) \cdot g_{\lambda, r}^{-1}(D_2)} &+ & g_{\lambda, r}^{-1}(D_1)\cdot \Sigma R_j & +  \\
& & & & & \\
\Sigma R_j \cdot \Sigma F_{b_i^2} &+ &
 \Sigma R_j\cdot g_{\lambda, r}^{-1}(D_2) &+ & \Sigma R_j \cdot \Sigma R_j
 \end{array}     \right.
$$

Next we make two observations. The first one is that 
$F_{b_i^k}\cdot  R_j=1$.
The second one is that, by the property (b) above,  the branching values of the restriction of  $g_{\lambda, r}$ to the fibres $F_{b_i^1}$ do not lie in $D_2$ which implies that 
 the divisors $F_{b_i^1}$ and $g_{\lambda, r}^{-1}(D_2)$ meet transversally at two points. It follows that
 $$ \Sigma F_{b_i^1}\cdot g_{\lambda, r}^{-1}(D_2)=
\Sigma F_{b_i^1}\cdot g_{\lambda, r}^{-1}(s_2)
+ \Sigma F_{b_i^1}\cdot g_{\lambda, r}^{-1}(t_2) 
=2(2 \gamma-2)+2(2 \gamma-2),$$
and the analogous statement  for $ g_{\lambda, r}^{-1}(D_1) \cdot \Sigma F_{b_i^2}$. 
This leaves us with 
\begin{equation} \label{K2}
K^2  = (2r+8)(2 \gamma-2)+\Sigma R_j^2 + \Sigma  g_{\lambda, r}^{-1}(D_1)\cdot R_j +  \Sigma R_j \cdot g_{\lambda, r}^{-1}(D_2)
\end{equation}
Now, Noether's formula for the genus of a curve embedded in a surface (see, for example, \cite[p.12]{Beau}) applied to 
$B \cong R_j \subset S_{\lambda, r}$
  yields the equality   
$$ R_j \cdot K= (2\gamma-2)- R_j^2  $$
But the intersection numbers $R_j \cdot K$ can also be expressed using again the representation of the canonical divisor $K=\mbox{div}(W_k)$, namely 
$$
\begin{array}{lllllll}
R_j \cdot K= R_j \cdot \mbox{div}(W_k)&=&R_j \cdot \Sigma F_{b_i^k} &+& R_j \cdot g_{\lambda,r}^{-1}(D_k) &+& R_j \cdot R \\
&=& (2\gamma-2) &+& R_j \cdot g_{\lambda,r}^{-1}(D_k) &+& R_j^2
\end{array}
$$
From these two expresions for $R_j \cdot K$ we infer that
$$R_j^2=-\frac{1}{2}R_j \cdot g_{ \lambda,r}^{-1}(D_1)=-\frac{1}{2}R_j \cdot g_{ \lambda,r}^{-1}(D_2);$$
 hence our  expression (\ref{K2}) for $K^2$ turns into 
\begin{equation*}
K^2  = (2r+8)(2 \gamma-2) +  \frac{3}{2} \Sigma R_j 
\cdot g_{\lambda, r}^{-1}(D_1).
\end{equation*}
It follows that in order to prove the identity \eqref{K22} it is enough to show that   
 $ \Sigma R_j 
\cdot g_{\lambda, r}^{-1}(D_1)=2(\gamma-1)$. Since $D_1=s_1+t_1$ all we have to do amounts to check that 
  $ R_j \cdot g_{\lambda, r}^{-1}(s_1)= g_{\lambda, r}^{-1}(t_1)\cdot R_j=(\gamma-1)/r$. Thus, the result will follow from the following  lemma.
 
\begin{lemma} 
\begin{enumerate} \mbox{}
\item $|g_{\lambda, r}^{-1}(s_1) \cap R_j|= |g_{\lambda, r}^{-1}(t_1) \cap R_j|=\dfrac{\gamma-1}{r}.$
\item The divisor $R_j$ meets transversally each of the divisors $g_{\lambda, r}^{-1}(s_1)$ and $g_{\lambda r}^{-1}(t_1)$.
\end{enumerate}
\end{lemma}

\begin{proof} 
 
(1) By (\ref{restrictiong}) $
f_{\lambda, r}$ induces  a bijection  between the finite sets 
$g_{\lambda,  r}^{-1}(s_1) \cap R_j$ and 
$(\pi_j \circ \zeta_{\lambda, r})^{-1}(s_1)$, its inverse being the map that sends each $b\in (\pi_j \circ \zeta_{\lambda, r})^{-1}(s_1)$ to the (unique) intersection point 
of $F_b$ and $R_j.$ Hence
 $|g_{\lambda,r}^{-1}(s_1^1) \cap R_j|=\mbox{deg}(\zeta_{\lambda, r}) |\pi_j^{-1}(s_1)|.$ Now, since only one of the $r$ coordinates of a 
point of $C_{\lambda, r}$ 
can be  equal to either  $s_1$ or $t_1,$ we infer
 that  these two points are  not  branch values of $\pi_j$ and that 
 $|\pi_j^{-1}(s_1)|=|\pi_j^{-1}(t_1)|=2^{r-1}= 
\mbox{deg}(\pi_j)$. Therefore 
$$|g_{\lambda, r}^{-1}(s_1) \cap R_j|=|g_{\lambda, r}^{-1}(t_1) \cap R_j|=2^{r-1}
\mbox{deg}(\zeta_{\lambda, r}).$$

Now, we apply the Riemann-Hurwitz formula to the 
 smooth covering $\zeta_{\lambda, r}:B\to  C_{\lambda, r}$
to see that $
2\gamma -2=2^r\mbox{deg}(\zeta_{\lambda, r}),$ and the result follows.

\vspace{0.2 cm}

 (2) To show that $R_j$ meets transversally the divisor given by the equation   $g_{\lambda,r}=s_1$ (resp. $g_{\lambda,r}= t_1$) at each point $P$ in the intersection, we must check that the restriction  of $g_{\lambda,r}$ to $R_j$ is non-singular at $P$. This follows at once from the expression $g_{\lambda,r}|_{R_j} =\pi_j\circ \zeta_{\lambda, r} \circ f_{\lambda,r}|_{R_j}   $ given in  (\ref{restrictiong}) since  
$f_{\lambda,r}|_{R_j}
:R_j \to B$ is an isomorphism, $\zeta_{\lambda, r}:B\to  C_{\lambda, r}$ is a smooth cover and, as noted above, any point of $C_{\lambda, r}$ 
in the fibre $\pi_j^{-1}(s_1)$ (resp. $\pi_j^{-1}(t_1)$),   such as
 $  \zeta_{\lambda, r} \circ f_{\lambda,r}|_{R_j}(P)$, is a regular point of $\pi_j$.
\end{proof} 

\begin{remark}
The referee has kindly pointed out to us that a shorter computation of $K^2$  can be achieved by using tools more specific of the theory of algebraic surfaces such as the ones employed by Catanese and Rollenske in \cite{Catanese}.
\end{remark}
 
\bibliographystyle{amsalpha}

\end{document}